\documentclass{amsart}
\usepackage{graphicx}
\usepackage{amssymb}

\newtheorem{theorem}{Theorem}[section]
\newtheorem*{THMX}{Main Theorem}
\newtheorem{proposition}[theorem]{Proposition}

\newtheorem{lemma}[theorem]{Lemma}

\theoremstyle{definition}

\providecommand{\A}{\mathcal A}

\title{Bounded type interval exchange maps}
\author{Dong Han Kim}
\address{Department of Mathematics Education, Dongguk University-Seoul, Seoul 100-715, Korea.}
\email{kim2010@dongguk.edu}
\thanks{This work was partially supported by the National Research Foundation of Korea(NRF) (2012R1A1A2004473).}
\author{Stefano Marmi}
\address{Scuola Normale Superiore, Piazza dei Cavalieri 7, 56126 Pisa, Italy}
\email{s.marmi@sns.it}


\begin{document}

\begin{abstract} 
Irrational numbers of bounded type have several equivalent characterizations. 
They have bounded partial quotients in terms of arithmetic characterization
and in the dynamics of the circle rotation, the rescaled recurrence time to $r$-ball of the initial point is bounded below.
In this paper, we consider how the bounded type condition of irrational is generalized into interval exchange maps.
\end{abstract}

\maketitle

\section{Introduction}

An irrational $\theta$ is said to be of bounded type (or constant type) if its partial quotients $(a_k)$ of the continued fraction expansion are bounded.
This condition is equivalent that there exists a constant $c > 0$ such that
$$ \liminf_{n \to \infty} n \cdot \| n\theta \| = \liminf_{n \to \infty} n \cdot |T_\theta^n(x) - x | \ge c,$$ 
where $\| \cdot \|$ denotes the distance to its nearest integer and $T_\theta$ is the rotation by $\theta$ on the unit circle.
In this paper, we investigate the bounded type condition for the interval exchange map.

An interval exchange map (i.e.m.) is determined by combinatorial data and length data.
The combinatorial data consists of a finite set $\A$ for the
subintervals and of two bijections $(\pi_t, \pi_b)$ from $\A$ onto $\{1,\ldots ,d\}$ ($|\A| = d$): these indicate
in which order the intervals are met before and after the map.
The length data  $(\lambda_\alpha)_{\alpha\in\mathcal{ A}}$ give the length
$\lambda_\alpha >0$ of the corresponding interval.

We set $$ 
p_\alpha := \sum_{\pi_t(\beta) < \pi_t (\alpha)} \lambda_\beta,
\quad q_\alpha := \sum_{\pi_b(\beta) < \pi_b(\alpha)} \lambda_\beta, 
\quad \lambda^* = \sum_\alpha \lambda_\alpha.$$
%I_\alpha = [p_\alpha, p_\alpha + \lambda_\alpha).$$
The i.e.m.\ $T$ associated to these data is defined as
$$ T(x) = x  - p_\alpha + q_\alpha \ \text{ for } x \in I_\alpha := [p_\alpha, p_\alpha + \lambda_\alpha).$$
%$$ T(x) = x + \sum_{\pi_b (\alpha) > \pi_b(\beta) } \lambda_\beta - \sum_{\pi_t (\alpha) > \pi_t(\beta) } \lambda_\beta \ \text{ for } x \in I_\alpha .$$

In the following, we will consider only combinatorial data $(\mathcal{ A} , \pi_t, \pi_b)$
which are {\it admissible}, i.e., 
$ \pi_t^{-1}(\{1,\ldots ,k\})\ne \pi_b^{-1}(\{1, \ldots ,k\})$ for $1 \le k < d$.
Moreover, we will assume our maps to have the {\it Keane property}:
if there exist $\alpha ,\beta \in \A$ and integer $m$ such that $T^m (p_\alpha)= p_\beta$ and $\pi_t (\beta )>1$.
The Keane property is the appropriate notion of irrationality for i.e.m.\cite{Ke1}.

For admissible interval exchange maps with the Keane property we can
introduce the generalization of continued fractions to i.e.m.'s
(see \cite{Y1, Y2} for a more detailed discussion) due to the
work of Rauzy \cite{Ra}, Veech \cite{Ve} and Zorich \cite{Z1,Z2}.

%Following \cite{MMY} we introduce here the basic notions about interval exchange maps needed in the sequel. 
%We also recall the construction and the fundamental properties of the continued fraction algorithm for interval exchange maps.
%We refer to \cite{MMY} and references therein for the proofs.

%Let $(\pi_t, \pi_b)$ be an admissible pair. We define two
%new admissible
%pairs $\mathcal{R}_t (\pi_t, \pi_b)$ and $\mathcal{R}_b (\pi_t, \pi_b)$ as follows:
%let $\alpha_t, \alpha_b$ be the (distinct) elements of $\mathcal{A}$ such that
%$\pi_t(\alpha_t)=\pi_b(\alpha_b)=d$; one has
%$$ \mathcal{R}_t (\pi_t, \pi_b) = (\pi_t, \hat{\pi}_b)\; , \qquad
%\mathcal{R}_b (\pi_t, \pi_b) = (\hat{\pi}_t, \pi_b)\; ,$$
%where
%$$\hat{\pi}_b (\alpha) = \begin{cases}
%\pi_b(\alpha ) &  \pi_b(\alpha )\le \pi_b(\alpha_t),\\
%\pi_b(\alpha )+1 &  \pi_b(\alpha_t )<\pi_b(\alpha)<d,\\
%\pi_b(\alpha_t )+1 & \alpha=\alpha_b, 
%\end{cases}
%\hat{\pi}_t (\alpha) = \begin{cases}
%\pi_t(\alpha ) & \pi_t(\alpha )\le \pi_t(\alpha_b),\\
%\pi_t(\alpha )+1 & \pi_t(\alpha_b )<\pi_t(\alpha)<d,\\
%\pi_t(\alpha_b )+1 & \alpha=\alpha_t.
%\end{cases} $$
%The {\it Rauzy class} of $(\pi_t, \pi_b)$ is the set of admissible pairs obtained by saturation of $(\pi_t, \pi_b)$ under the action of $\mathcal{R}_t$ and $\mathcal{R}_b$.

We say that $T$ is of {\it top type} (respectively {\it bottom type})
if one has $\lambda_{\alpha_t}\ge\lambda_{\alpha_b}$ (respectively $\lambda_{\alpha_b}\ge\lambda_{\alpha_t}$);
we then define a new i.e.m.\ $\mathcal{V} (T)$ as the induced map on $\left[0, \lambda^* - \lambda_{\alpha_b}\right)$
(respectively $\left[0, \lambda^* - \lambda_{\alpha_t}\right)$),
which is given by a new admissible pair $\mathcal{R}_t(\pi_t, \pi_b)$ and the lengths
$(\hat{\lambda}_\alpha)_{\alpha\in\mathcal{A}}$ given by
$$\begin{cases}
\hat{\lambda}_\alpha = \lambda_\alpha & \textrm{if } \alpha\not=\alpha_t,\\
\hat{\lambda}_{\alpha_t} = \lambda_{\alpha_t}- \lambda_{\alpha_b} & \textrm{otherwise}
\end{cases}$$
for the top type $T$; a new admissible pair $\mathcal{R}_b(\pi_t, \pi_b)$ and the lengths $(\hat{\lambda}_\alpha)_{\alpha\in\mathcal{A}}$ given by
$$\begin{cases}
\hat{\lambda}_\alpha = \lambda_\alpha & \textrm{if } \alpha\not=\alpha_b,\\
\hat{\lambda}_{\alpha_b} = \lambda_{\alpha_b}- \lambda_{\alpha_t} & \textrm{otherwise}
\end{cases}$$
for the bottom type $T$.

The {\it Rauzy diagram} is the graph of vertices obtained by saturation of $(\pi_t,\pi_b)$ under the action of $\mathcal R_t$ and $\mathcal R_b$ and two arrows joining $(\pi_t,\pi_b)$ to $\mathcal{R}_t (\pi_t, \pi_b)$, $\mathcal{R}_b (\pi_t, \pi_b)$.
For an arrow joining $(\pi_t,\pi_b)$ to $\mathcal{R}_t (\pi_t, \pi_b)$ (respectively $\mathcal{R}_b (\pi_t, \pi_b)$) the element $\alpha_t \in \A$ (respectively $\alpha_b \in \A$) is called the {\it winner} and the element $\alpha_b \in \A$ (respectively $\alpha_t \in \A$) is called the {\it loser}.
%The i.e.m.\ $\mathcal{V} (T)$ is the first return map of $T$ on $\left[0,\sum_{\alpha}\hat \lambda_\alpha \right)$.
%We also associate to $T$ the arrow in the Rauzy diagram joining $(\pi_t,\pi_b)$ to $\mathcal{R}_t(\pi_t, \pi_b)$ or $\mathcal{R}_b(\pi_t, \pi_b)$ depending the type of $T$.

Iterating this process, we obtain a sequence of i.e.m.\ $T(n) = \mathcal{V}^n (T)$ and an infinite path in the Rauzy diagram starting from $(\pi_t,\pi_b)$. 
In fact, a further property of irrational interval exchange maps (i.e.\ with the Keane property) is that every letter is taken infinitely many times as the winner
in the infinite path (in the Rauzy diagram) associated to $T$.
This property is fundamental in order to be able to group together several iterations of $\mathcal{V}$ to obtain the accelerated Zorich continued fraction algorithm.

For an arrow $\gamma$ with winner $\alpha$ and loser $\beta$ in the Rauzy diagram let $B_\gamma = \mathbb I + E_{\beta\alpha}$ where $\mathbb I$ is the identity matrix and $E_{\beta\alpha}$ is the matrix with only non-zero entry 1 at $(\beta,\alpha)$.
For a finite path $\underline{\gamma} = \gamma_1 \gamma_2 \cdots \gamma_n$ in the Rauzy diagram we associate a $SL(\mathbb Z^{\A})$ matrix with non-negative entries 
$$ B_{\underline{\gamma}} = B_{\gamma_n} \cdots B_{\gamma_1}.$$
Let $\gamma^T(m,n) = \gamma(m,n)$, $m \le n$ be the path in the Rauzy diagram from the permutation of $T(m)$ to the permutation of $T(n)$
and denote 
$$B(m,n) = B_{\gamma (m,n)}, \qquad B(n) = B(0,n).$$
Let $\lambda(n)$ be the length data of $T(n)$. Then we have 
\begin{equation}\label{eqn1}
\lambda(m) = \lambda(n) B(m,n) .
\end{equation}

Zorich's accelerated continued fraction algorithm is obtained by considering 
by $(\mathcal{V}^{n_k})_{k\ge 0}$
 where $n_k$ is the following sequence: $n_0 = 0$ and $n_{k+1} > n_k$ is chosen so as to assure that $\gamma(n_k,n_{k+1})$ is the longest path whose arrows have the same winner.
 
A further acceleration algorithm by Marmi-Moussa-Yoccoz \cite{MMY} is obtained by $(\mathcal{V}^{m_k})_{k\ge 0}$
 where $m_k$ is defined as follows: $m_0 = 0$ and $m_{k+1} > m_k$ is the largest integer such that all letters in $\mathcal{A}$ are taken as winner by arrows in $\gamma(m_k,m_{k+1})$.

%For $m<n$, $T(n)$ is the first return map of $T(m)$on $I(n)=\left[0,\sum_{\alpha\in\mathcal{A}} \lambda_\alpha (n) \right)$;
%the return time of $I_\beta (n)$ in $I(n)$ is $Q_{\beta} (m,n) :=  \sum_\alpha Q_{\alpha\beta}(m,n)$ and the time spent in $I_\alpha (m)$ is $Q_{\alpha\beta}(m,n)$,
%where
%$$ I_\beta (n) = [ 0, \lambda_\beta (n)) + p_\alpha(n).$$
%$$I(n) = \bigsqcup_{\beta \in \mathcal A} I_\beta (n).$$

%-----------------------------------------------

%There are several ways to define bounded type for an interval exchange map $T$:
As the irrational rotations,
bounded type interval exchange map $T$ can be characterized by its continued fraction matrix.

\begin{itemize}

\item[(A)] The MMY cocycle matrices is bounded, i.e., 
$$ \| B(m_k, m_{k+1}) \| \le M. $$ 

\item[(Z)] The Zorich cocycle matrices are bounded, i.e.,  
$$ \| B(n_k, n_{k+1}) \| \le M. $$ 
\end{itemize}

Let $\Delta (T)$ be the minimal distance between discontinuities of $T$.
We have the following characterization of the bounded type i.e.m. by dynamics of $T$.

\begin{itemize}
\item[(D)]
There is a constant $c$ such that 
$$\Delta (T^n) \ge \frac{c}{n} \text{ for all } n.$$

\item[(U)]
There is a constant $c >0$ such that for all $x$
$$ \liminf_{n \to \infty} n \cdot | T^n (x) - x| \ge c. $$ 

%\item[(R)]
%For every $x$
%$$ \liminf_{n \to \infty} n \cdot | T^n (x) - x| > 0. $$ 

%\item[(G)]
%The Teichm\'uller geodesic flow is bounded.
\end{itemize}
Here and after, the matrix norm is $\| A \|= \sum_{i,j}|a_{ij}|$ for a matrix $A = (a_{i,j})$.

\begin{THMX}
(i) (A) bounded type and (D) bounded type are equivalent.

(ii) (D) bounded type implies (U) bounded type.

(iii) (U) bounded type implies (Z) bounded type.
\end{THMX}

The inverse of (ii) and (iii) do not hold. See Section~\ref{exm} for the examples.
The proof of (i) is given in Section 2 and In Section 3, the proofs of (ii) and (iii) are presented.

In the last ten years, there has been progress in diophanitne condition of the i.e.m. (see \cite{BC, Cha, Mar})
and the Roth type diophantine condition for the i.e.m. has been studied in \cite{Dio, KimMarmi, MMY2}. 
Condition (D) is considered by Boshernitzan for unique ergodicity\cite{Bo}.
Condition (D) and the bounded minimum saddle connection are equivalent, see \cite{HMU, Vo}.

%In \cite{HMU}, Hubert, Marchese and Ulcigrai found bounds for minimum saddle connection and norm of minimal positive continued fraction matrices, which is essentially not different from condition (A).

While we were writing this paper, we found out that the equivalence (i) in the Main Theorem is also proved in the recent preprint \cite{HMU} by Hubert, Marchese and Ulcigrai. 
In \cite{HMU} the authors consider a related acceleration of the algorithm (the positive acceleration), whose matrices are uniformly bounded if and only if (A) holds and show that bounded positive matrices is equivalent to (D) (see \cite{HMU}, Corollary 4.8 combined with Proposition 1.1).  
From their proof one can also obtain quantitative relations between the constant $c$ in (D) and the norm of the matrices (see Theorem 4.7 in \cite{HMU}), but our proof is much shorter and less combinatorially involved than theirs (compare with Appendix C in \cite{HMU}).

\section{The bounded MMY cocycle condition and the bounded gap condition of discontinuities}

In this section, we prove (i) of the main theorem through Proposition~\ref{prop1} and \ref{prop2}.
Let $r = \max(2d - 3, 2)$. By \cite{MMY}, each entry of the matrix $B(m_k, m_{k+r})$ is strictly positive.

Assume that 
$\| B(m_k, m_{k+1}) \| \le M $ for all $k$.
Then clearly we have 
\begin{equation}\label{basic}
\| B(m_k, m_{k+r}) \| \le M^r, \qquad \| B (m_{k+r}) \| \le M^r \| B (m_k) \|.
\end{equation}

Since 
$$ \sum_{\alpha \in \A} \lambda_\alpha (m_{k+r}) B_{\alpha \beta } (m_k, m_{k+r}) = \lambda_\beta (m_k)$$
and $B_{\alpha \beta} (m_k, m_{k+r}) \ge 1$ for all $\alpha, \beta \in \mathcal A$, 
we have
\begin{equation}\label{minmax}
\min_{\alpha \in \A} \lambda_\alpha (m_k) > \max_{\alpha \in \A} \lambda_\alpha (m_{k+r}).
\end{equation}

Also we have 
$$ \sum_{\alpha, \beta \in \A} \lambda_\alpha (m_k) B_{\alpha \beta } (m_k) = 1,$$
which follows 
\begin{equation}\label{maxq}
 \| B(m_k) \| \cdot \max_{\alpha \in \A} \lambda_\alpha (m_k) \ge 1.
\end{equation}

Let $B_\alpha (m_k) = \sum_\beta B_{\alpha \beta} (m_k)$.
Then for any $\alpha \in \A$ we have
\begin{equation}\label{minq}
\| B(m_k) \| = \sum_{\beta, \delta} B_{\beta \delta} (m_k) \le \sum_{\beta, \delta} B_{\alpha \delta} (m_k, m_{k+r}) B_{\delta \beta} (m_k) = B_\alpha (m_{k+r}).
\end{equation}

%Let $p_\alpha$ be the discontinuities of $T$ (except for $\alpha$ such that $\pi_t (\alpha) = 1$)

\begin{lemma}[\cite{Dio}, Lemma 4.2 and 4.3]\label{lem_Dio}
If $0 < n \le \min_\alpha B_\alpha (m_k)$, then 
$$ \min_{\alpha \in \A} \lambda_\alpha (m_{k+r}) \le \Delta (T(m_k)^2) \le \Delta(T^n).$$
%$$ D(T^n) \subset \bigcup_{p \in D(T(m_k)^2)} \{ T^i(p) : 0 \le i < B_\alpha (m_k),  p \in I_\alpha(m_k) \}.$$
\end{lemma}

%\begin{lemma}[\cite{Dio}, Lemma 4.3]
%$$ \min_{\alpha \in \A} \lambda_\alpha (m_{k+r}) \le \Delta (T(m_k)^2)$$
%\end{lemma}

For any $\min_\alpha B_\alpha (m_{k-1}) < n \le \min_\alpha B_\alpha (m_k)$ 
we have
\begin{align*}
\Delta(T^n) &\ge \min_\alpha \lambda_\alpha (m_{k+r}), &\text{ by Lemma~\ref{lem_Dio}},\\
&\ge \max_\alpha \lambda_\alpha (m_{k+2r}), &\text{ by (\ref{minmax})},\\
&\ge \frac{1}{\|B(m_{k+2r})\|}, &\text{ by (\ref{maxq})}, \\
&\ge \frac 1{M^{4r}} \cdot \frac{1}{\|B(m_{k-2r})\|}, &\text{ by (\ref{basic})}, \\
&> \frac 1{M^{4r}} \cdot \frac{1}{\min_\alpha B_\alpha( m_{k-r} )}, &\text{ by (\ref{minq})}, \\
&> \frac 1{M^{4r}} \cdot \frac{1}{n} &\text{ by the assumption}.
\end{align*}

Therefore, we have the following:
\begin{proposition}\label{prop1}
If the MMY cocycle of $T$ satisfies $\|B(m_k,m_{k+1})\| \le M$, then we have
$$ \Delta(T^n) > \frac{1}{M^{4r}} \cdot \frac 1n .$$
\end{proposition}

%Note that $T(k)$ be the induced map of $T$ on $[0,\lambda^*(k))$.

\begin{lemma}\label{nnn}
We have either
$$\min_\alpha \lambda_\alpha (m_k) \cdot \sqrt{\| B(m_k,m_{k+1}) \|} <  \lambda^* (m_k) $$
or 
$$\min_\alpha \lambda_\alpha (m_{k+1}) \cdot \sqrt{\| B(m_k,m_{k+1}) \|} <  \lambda^* (m_{k+1}) $$
\end{lemma}

\begin{proof}
For each $k$ let $\alpha = \alpha(k) \in \mathcal A$, depending on $k$, be the letter which is not taken as the winner of the arrows in the path $\gamma(m_k,m_{k+1})$.
Then
$$ \lambda_\alpha (m_k) = \lambda_\alpha (m_{k+1}) .$$

Now we have two cases:

Case (i) :
$\lambda_\alpha (m_k) \cdot \sqrt{\| B(m_k,m_{k+1}) \|} <  \lambda^* (m_k) $ , which implies the lemma. \\
 
Case (ii) :
$\lambda_\alpha (m_k) \cdot \sqrt{\| B(m_k,m_{k+1}) \|} \ge  \lambda^* (m_k) $ \\
Since
$$\sum_{\alpha , \beta \in \mathcal A} \lambda_\alpha (m_{k+1}) B_{\alpha \beta} (m_k,m_{k+1}) = \lambda^* (m_k),$$
we have
$$ \min_\alpha \lambda_\alpha (m_{k+1}) \cdot \| B(m_k,m_{k+1}) \| <  \lambda^* (m_k) < \max_\alpha \lambda_\alpha (m_{k+1}) \cdot \| B(m_k,m_{k+1}) \| . $$
Thus, there exists $\beta \in \mathcal A$ such that 
$$ \lambda_\beta (m_{k+1})  \cdot \| B(m_k,m_{k+1}) \| <  \lambda^* (m_k) \le
\lambda_\alpha (m_k) \cdot \sqrt{\| B( m_k, m_{k+1}) \|} .$$
Therefore, we have
\begin{equation*} 
\lambda^* (m_{k+1}) > \lambda_\alpha (m_{k+1}) = \lambda_\alpha (m_k) > \lambda_\beta (m_{k+1}) \cdot \sqrt{\| B(m_k, m_{k+1}) \|}. \qedhere
\end{equation*}
\end{proof}

\begin{lemma}\label{new}
Let $\alpha \in \mathcal A$ be the winner of $\gamma(n-1,n)$ and the loser of $\gamma(n,n+1)$.
For large $n$, if $\lambda_\alpha (n) <  \lambda^* (n)/M^d$, $M > d$, 
then there is an integer $s$, $1 \le s < d$, such that
$$ \Delta \left(T^{\lfloor 2 M^s/ \lambda^* (n) \rfloor} \right) < (d-1) \frac{\lambda^*(n)}{M^{s+1}}. $$
\end{lemma}

\begin{proof}
Let for $0 \le i < d $
$$
\mathcal A_i = \left \{ \beta \in \mathcal A : \frac{\lambda^* (n)}{M^{i+1}} \le \lambda_\beta (n) < \frac{\lambda^* (n)}{M^{i}} \right \}
$$
and
$$
\mathcal A_d = \{ \beta \in \mathcal A : \lambda_\beta (n) < \frac{\lambda^* (n)}{M^d} \}.
$$
Then, by the assumption, $\alpha \in \mathcal A_d \ne \emptyset$.
Since there is an $\beta \in \mathcal A$ such that 
$\lambda_\beta (n) > \lambda^* (n)/{d} >  \lambda^* (n)/M$, neither $\mathcal A_0$ is an empty set.

Since there are $d$ elements in $\mathcal A$, there exist an $s$, $1 \le s < d$, such that $\mathcal A_s $ is empty.
Let 
$$
\mathcal A_{\textrm{big}} = \bigcup_{i=0}^{s-1} \mathcal A_i, \qquad
\mathcal A_{\textrm{small}} = \bigcup_{i=s+1}^{d} \mathcal A_i. 
$$
Both $\mathcal A_{\textrm{big}}$ and $\mathcal A_{\textrm{small}}$ are nonempty.

Take $m$ with $m < n$ be the smallest integer as 
no loser in $\gamma(m+1,n)$ belongs to $\mathcal A_{\textrm{big}}$. 
Put $\mu \in \mathcal A_{\textrm{big}}$ as the loser of the arrow $\gamma(m,m+1)$. 
Let $\nu$ be the winner of the arrow $\gamma(m,m+1)$.
Then $\nu \in \mathcal A_{\textrm{small}}$.
(if $\nu \in \mathcal A_{\textrm{big}}$, then $\nu \ne \alpha$ and $\nu$ should be a loser in $\gamma(m+1,n)$)

Hence we have
$\lambda_{\nu} (m+1) =  \lambda_{\nu} (m) - \lambda_{\mu} (m)$ and
\begin{align*}
B_{\nu} (m+1) &= B_{\nu} (m) < \frac{1}{\lambda_\nu (m)} < \frac{1}{\lambda_{\mu}(m)} \le \frac{M^s}{ \lambda^* (n)}, \\
B_{\mu} (m+1) &= B_{\nu} (m) + B_{\mu} (m)< \frac{1}{\lambda_\nu (m)} + \frac{1}{\lambda_{\mu}(m)} < \frac{2}{\lambda_{\mu}(m)} \le \frac{2 M^s}{ \lambda^* (n) }.
\end{align*}

There are two cases:

(i) $\pi_t^{(m)}(\mu ) = d$ and $\pi_b^{(m)}(\nu) = d$: \\
Then we have $\pi_t^{(m+1)}(\nu ) = \pi_t^{(m)}(\nu ) < d$ and 
$\pi_t^{(m+1)}(\mu ) = \pi_t^{(m)}(\nu ) +1$.

Since no letter in $\mathcal A_{\textrm{big}}$ is taken as the winner or the loser of the arrows of $\gamma(m+1,n)$, 
$$ I_\nu (m) = I_\nu (m+1) \cup \ I_\mu (m+1) \subset [0, \lambda^*(n))$$
and
$$ I_\nu (m+1) = \left[ p_\nu (m+1) , p_\mu (m+1) \right ).$$
Since $p_\nu (m+1)$, $p_\mu (m+1)$ are discontinuity points of $T(n)$ and $\pi_b^{(m+1)}(\nu) = d$, $m+1 \le n$,
we have $$I_\nu (m+1) = \bigsqcup_{\beta \in \mathcal A'} I_\beta (n)  
\ \text{ for some } \mathcal A' \subset A_{\textrm{small}}.$$
Therefore, we have 
\begin{equation*}
p_\mu (m+1)-p_\nu(m+1) =\lambda_\nu(m+1) 
< | \mathcal A_{\textrm{small}} | \cdot \frac{\lambda^*(n)}{M^{s+1}} 
\le (d-1)\frac{\lambda^*(n)}{M^{s+1}}.
\end{equation*}
Since 
\begin{align*}
p_\mu (m+1) &\in D \left( T^{B_\mu(m+1)} \right) = D \left( T^{B_\nu(m) + B_\mu(m)} \right). \\
p_\nu (m+1) &\in D \left( T^{B_\nu(m+1)} \right) = D \left( T^{B_\nu(m) } \right),
\end{align*}
we have
$$ p_\mu (m+1) - p_\nu (m+1) \ge \Delta \left( T^{B_\mu(m+1)} \right) .$$

(ii) $\pi_t^{(m)}(\nu ) = d$ and $\pi_b^{(m)}(\mu ) = d$: \\
Then we have $\pi_b^{(m+1)}(\nu ) = \pi_b^{(m)}(\nu ) < d$ and 
$\pi_b^{(m+1)}(\mu ) = \pi_b^{(m)}(\nu ) +1$.
Similarly with case (i), we have
$$q_\mu (m+1) - q_\nu (m+1) = \lambda_\nu (m+1) < (d-1) \frac{\lambda^*(n)}{M^{s+1}}.$$
Since 
\begin{align*}
q_\mu (m+1) &\in D \left( T^{-B_\mu(m+1)} \right) = D \left( T^{-B_\nu(m) - B_\mu(m)} \right) ,\\
q_\nu (m+1) &\in D \left( T^{-B_\nu(m+1)} \right) = D \left( T^{-B_\nu(m) } \right) ,
\end{align*} we have
$$ q_\mu (m+1) - q_\nu (m+1) \ge \Delta \left ( T^{- B_\mu(m+1)} \right )
= \Delta \left ( T^{B_\mu(m+1)} \right ) .$$
Note that $\Delta(T) = \Delta(T^{-1})$.
\end{proof}

\begin{proposition}\label{prop2}
If $\limsup_{k \to \infty} \| B(m_k,m_{k+1}) \| = \infty$, then $\liminf_{n \to \infty} n \cdot  \Delta (T^n) = 0$.
\end{proposition}

\begin{proof}
For any give big $M > 0$, by Lemma~\ref{nnn}, there are infinitely many $k$ and $\alpha$ (depending on $k$) satisfying 
$$ \lambda_{\alpha}(m_k)= \min_{\beta \in \mathcal A} \lambda_{\beta}(m_k) < \frac{\lambda^* (m_k)}{M}.$$

Let $\ell_k(\alpha) = \max \{n \le m_k: \alpha \text{ is the winner of } \gamma(n-1,n)\}$. 
Since $ \lambda_{\alpha}(m_k)= \min_{\beta \in \mathcal A} \lambda_{\beta}(m_k)$,
$\alpha$ cannot be the winner of $\gamma(m_k,m_k +1)$. 
Thus $\alpha$ should be the winner of an arrow in $\gamma(m_{k-1},m_k)$, which yields
$$ m_{k-1} < \ell_k(\alpha) < m_k.$$ 
Thus,
$$ \lambda_{\alpha}(\ell_k(\alpha)) = \lambda_{\alpha}(m_k) < \frac{\lambda^* (m_k)}{M} \le \frac{\lambda^* (\ell_k(\alpha))}{M}.$$

Hence, we can choose infinitely many $n$'s satisfying the condition for Lemma~\ref{new}, which completes the proof. 
\end{proof}

\section{Bounded Zorich cocycle condition and the uniform recurrence condition}

In this section, we prove (ii) and (iii) of the main theorem.

The proof of (ii) and (iii) in the main theorem are directly obtained by the following propositions:

\begin{proposition}
If $n \cdot| T^n (x) - x | < c$ for some $x$, then we have $$\Delta(T^{2n}) < \frac{c}{n}$$
\end{proposition}

\begin{proof}
Let $[a,b)$ be the maximal interval containing $x$ on which $T^n$ is continuous. 
If $b-a < c/n$, then the proof is completed.

Suppose that $b-a \ge \frac cn$.
Let $\delta = T^n (x) - x$. Then
$|\delta| < \frac cn \le b-a$ and $T^n [a,b) = [a+\delta, b+ \delta)$.
Therefore, either $T^n(a-\delta) = a$ or $T^n(b-\delta) = b$ is a discontinuous point of $T^n$ depending on $\delta < 0$ or $\delta > 0$.
Hence, either $a-\delta$ or $b-\delta$ is a discontinuous point of $T^{2n}$,
which implies that $$\Delta(T^{2n}) \le |\delta| <\frac cn$$
since $a$ and $b$ are discontinuous point of $T^n$ or the end point of the interval.
\end{proof}

\begin{proposition}
If $n_{k+1} - n_k \ge d-1$ and $\| B(n_k, n_{k+1}) \| > 2d$, then
there exist $x$ and $m$ such that
$$
m \cdot | T^m (x) - x| < \frac{d}{\| B(n_k, n_{k+1}) \| - 2d}
$$
\end{proposition}

\begin{proof}
Let $\alpha \in \mathcal A$ be the winner of the arrows and $\mathcal A'$ be the set of the losers of the arrows in the path $\gamma(n_k, n_{k+1})$.
If $\pi^{(n_k)}_t (\alpha) = d$, then 
$ \mathcal A' = \{ \beta \in \mathcal A : \pi^{(n_k)}_b (\beta) > \pi^{(n_k)}_b (\alpha) \}$
and $\pi^{(n)}_b$ is the cyclic permutation on $\mathcal A'$ for $n_k \le n \le n_{k+1}$.
For each $\beta \in \mathcal A'$ put $h_\beta = B_{\beta \alpha}(n_k,n_{k+1})$,
the number of arrows, of which loser is $\beta \in \mathcal A'$, in the path $\gamma(n_k, n_{k+1})$.
Put $$h := \left \lfloor \frac{n_{k+1} - n_k}{|\mathcal A'|} \right \rfloor \ge 1.$$ 
Then $h \le h_\beta \le h+1$ for all $\beta \in \mathcal A'$ and
\begin{equation*}
\| B(n_k,n_{k+1}) \| = d + n_{k+1} - n_k \le d + (h+1) \cdot |\mathcal A'|
< d(h +2). \end{equation*}
Let 
$$\frac{\lambda_{\alpha} (n_k)}{h}
> \frac{\lambda_{\alpha} (n_k) - \lambda_{\alpha} (n_{k+1})}{h}
= \frac{\sum_{\beta \in \mathcal A'} h_\beta \lambda_\beta(n_k)}{h}
\ge \sum_{\beta \in \mathcal A'} \lambda_\beta(n_k). $$

Let $m = B_{\alpha} (n_k)$. Then on $ x \in I_{\alpha} (n_k)$ we have
$$ | T^m (x) - x | =  \sum_{\beta \in \mathcal A'} \lambda_\beta(n_k)
< \frac{\lambda_{\alpha} (n_k)}{h}.$$
Hence we have for $x \in I_{\alpha} (n_k)$
$$ m \cdot | T^m (x) - x | < B_{\alpha} (n_k) \cdot \frac{\lambda_{\alpha} (n_k)}{h} < \frac{B_{\alpha} (n_k) \lambda_{\alpha} (n_k)}{ \| B (n_k,n_{k+1}) \|/d -2} < \frac{1}{ \| B (n_k,n_{k+1}) \|/d -2}.$$

For the case $\pi^{(n_k)}_b (\alpha) = d$ we have the same procedure.
\end{proof}

\section{Examples}\label{exm}

The example satisfying condition (U) without (D) can be obtained by 3-interval exchanges.
Let $T$ be a 3 interval exchange map with permutation $\left( \begin{smallmatrix} A & B & C \\ C & B & A \end{smallmatrix} \right)$ and the length data $(\lambda_A, \lambda_B, \lambda_C)$.
Then $T$ is the induced transformation of the translation 
$$\bar T(x) = \begin{cases} 
x + \lambda_B + \lambda_C, & x + \lambda_B + \lambda_C < 1+ \lambda_B, \\
x - \lambda_A - \lambda_B, & x + \lambda_B + \lambda_C \ge 1+ \lambda_B \end{cases} $$
on the unit interval $[0,1)$.
Then condition (U) or (Z) are equivalent that $\bar T$ is of bounded type.

Let $T$ has the infinite path in the Rauzy diagram given by sequence of the winners of the arrows as follows 
$$ ABA (A^2 C^2)^{n_1} ABA (A^2 C^2)^{n_2} ABA (A^2 C^2)^{n_3} \cdots $$
with $n_k \to \infty$. 
Clearly, $T$ does not satisfies condition $(A)$.
However, the infinite path in the Rauzy diagram for $\bar T$ is 
$$
\bar A \bar C \bar A \bar C \bar A \bar C \bar A \bar C \cdots,
$$
if we consider $\bar T$ as the 2-interval exchange map with permutation $\left( \begin{smallmatrix} \bar A & \bar C \\ \bar C & \bar A \end{smallmatrix} \right)$ and the length data $(\lambda_{\bar A}, \lambda_{\bar C}) =  (\lambda_A + \lambda_B, \lambda_C + \lambda_B)$. 
Therefore, $\bar T$ is the rotation by the golden mean $g^{-1} = (\sqrt{5}-1)/2$. 
It follows that $T$ satisfies condition $(U)$. 

An example satisfying condition (Z) but not condition (U) is explained in \cite{Dio}: 
Let $T$ be the 4-interval exchange map with permutation data $\left( \begin{smallmatrix} A & B & D & C \\ D & A & C & B \end{smallmatrix} \right)$ with infinite path in the Rauzy diagram 
$$ CB^3 ( D^2 A^3 D)^{2^1}B \cdot CB^3 ( D^2 A^3 D)^{2^2}B \cdots
CB^3 ( D^2 A^3 D)^{2^k}B \cdots .$$
Then $T$ satisfies condition (Z) since there is no long sequence of the same winner.
However, it was shown in \cite[Section 9]{Dio} that for $x \in I_C(n)$, $n = \sum_{i=1}^k (5 + 6 \cdot 2^i) + 3 = 2k + 12 (2^k-1) + 3$, 
$$ n \cdot | T^n (x) - x| < \frac{2g^{2^{k+2}+5k}}{g^{2^{k+3}+k-4}}, $$
where $g = \frac 12 (\sqrt 5 + 1)$.

\section*{acknowledgement}
The authors wish to thank Corinna Ulcigrai and Luca Marchese for introducing their works and very helpful comments.

\end{document}